\documentclass[a4paper,12pt]{article}
\usepackage[utf8]{inputenc}
\usepackage{amsmath}
\usepackage{amssymb}
\usepackage{amsthm}
\usepackage{appendix}

\usepackage{fullpage}

\newtheorem{thm}{Theorem}[section]

\newtheorem{cor}{Corollary}[section]
\newtheorem{prop}{Proposition}[section]
\newtheorem{remk}{Remark}[section]

\numberwithin{equation}{section}

\title{The Riemann problem for a generalised Burgers equation with spatially decaying sound speed. II General qualitative theory}
\author{J. C. Meyer and D. J. Needham}
\date{\today}

\begin{document}

\maketitle

\begin{abstract}
    \noindent We establish that the initial value problem for a generalised Burgers equation considered in part I of this paper, \cite{DNJBJMCD1}, is well-posed.
    We also establish several qualitative properties of solutions to the initial value problem utilised in \cite{DNJBJMCD1}.
\end{abstract}

$ $



\section{Introduction}
In this paper, we establish that the initial value problem for the generalised Burgers equation considered in \cite{DNJBJMCD1}, is well-posed (specifically, see Theorem \ref{GWP}).
To establish the existence result, we adopt the approach used in \cite{HZ1}, and note, that related standard existence results for classical solutions in \cite{AF1}, and similar sources, cannot be applied due to insufficient regularity of solutions to the initial value problem as $t\to 0^+$. 
The approach is centered on establishing sufficient regularity on solutions to an implicit integral equation, to establish that they are equivalent to classical solutions to the initial value problem.
We subsequently establish uniqueness and continuous dependence results for solutions to the initial value problem, via maximum principles in \cite{JMDN1}, and, the Gr\"onwall inequality in \cite{HYJGYD1}, respectively. 

\subsection{The Initial Value Problem}

Let $T>0$, $D_T= \{ (x,t) \in \mathbb{R}\times (0,T] \}$ and $\partial D = \bar{D}_T\setminus D_T$. 
We consider the Cauchy problem for $u: D_T\to\mathbb{R}$ given by:

\begin{align}
\label{IVP1} 
& u\in L^\infty (D_T) \cap C^{2,1} (D_T); \\
\label{IVP2}
& u_t - u_{xx} + h_\alpha(x) uu_x = 0 \quad \text{ on } D_T; \\
\label{IVP2'}
& h_\alpha (x) = \frac{1}{(1+x^2)^\alpha} \quad \forall\ x\in\mathbb{R}; \\
\label{IVP3''}
& u(x,t) = \int_{-\infty}^\infty \frac{u_0 (s)}{\sqrt{4\pi t}} e^{-\frac{(x-s)^2}{4t}} ds + O(t) \text{ uniformly for } x\in\mathbb{R} \text{ as } t\to 0^+.
\end{align}
Here $\alpha \in \mathbb{R}^+$ and $u_0:\mathbb{R}\to\mathbb{R}$ is the prescribed initial data, which is Lebesgue measurable, with $u_0\in L^\infty(\mathbb{R})$.
We denote the Cauchy problem given by \eqref{IVP1}-\eqref{IVP3''} as [IVP].
It should be noted that, via \eqref{IVP3''}, at each point $x\in\mathbb{R}$ at which $u_0$ is continuous, then $u(x,t)\to u_0(x)$ as $t\to 0^+$. 
Moreover, when $u_0$ is continuous for $x \in [x_1,x_2]$, then $u(x,t)\to u_0(x)$ uniformly for $x\in [x_1,x_2]$.
We later consider the specific case of [IVP] with $u_0:\mathbb{R}\to\mathbb{R}$ given by, 
\begin{equation}
\label{IVP3}
u_0(x) =  
\begin{cases} 
u^+, & x > 0, \\ 
u^-, & x \leq 0;
\end{cases}
\end{equation}
for $u^+,u^- \in \mathbb{R}$. 
For initial data given by \eqref{IVP3}, observe that we can replace \eqref{IVP3''} with $u:\bar{D}_T\to\mathbb{R}$, $u=u_0$ on $\partial D$, and $u \in C(\bar{D}_T\setminus \{ (0,0) \} )$.

\section{Qualitative Properties of [IVP]}

We introduce the fundamental solution to the heat equation on $D_T$, as $G: \mathcal{X}_T \to \mathbb{R}$, given by
\begin{equation}
\label{Gdef}
G(x,t;s,\tau ) = \frac{1}{\sqrt{4\pi (t-\tau )}}e^{-\frac{(x-s)^2}{4(t-\tau )}} \quad \forall\ (x,t;s,\tau ) \in \mathcal{X}_T
\end{equation} 
with $\mathcal{X}_T = \{ (x,t;s,\tau ) : (x,t) \in D_T,\ (s,\tau) \in \bar{D}_T,\ \tau < t \}$. 
Properties of $G$ which are used to establish the existence of solutions to [IVP] are given in Appendix \ref{Appendix-G}.

To establish global existence and uniqueness of solutions to [IVP], and local well-posedness in time, we consider an alternative to [IVP]. 
By applying a Duhamel principle, it follows that if $u:D_T\to\mathbb{R}$ is a solution to [IVP] then 
\begin{align}
\nonumber
& u(x,t) = \int_{-\infty}^\infty u_0(s) G(x,t;s,0) ds \\ 
\nonumber
& \quad \quad \quad \quad  +  \int_{0}^t \int_{-\infty}^\infty \frac{u^2(s,\tau )}{2} \left( G(x,t;s,\tau )h_\alpha '(s) \right. \\ 
\label{IE1}
& \left. \quad \quad \quad \quad \quad \quad \quad \quad \quad + G_s(x,t;s,\tau ) h_\alpha (s) \right) ds d\tau \quad \forall (x,t) \in D_T, \\
\label{IE2} 
& u\in C(D_T)\cap L^\infty (D_T).
\end{align} 
We will now demonstrate that there exists a local solution to \eqref{IE1} and \eqref{IE2}. 
The existence and regularity results for solutions to \eqref{IE1} and \eqref{IE2}, follow a similar approach to that developed in \cite{HZ1}. 
To begin, we have

\begin{prop} \label{locex}
The problem given by \eqref{IE1} and \eqref{IE2} has a solution $u:D_{T^*}\to\mathbb{R}$ with
\begin{align}
\nonumber
T^*= T(||u_0||_\infty , \alpha ) = 
& \min \left\{ 
1, \ 
\left( (||u_0||_\infty + 1)^2 \left( \frac{||h_\alpha'||_\infty}{2} + \frac{1}{\sqrt{\pi}} \right) \right)^{-2} , \ \right. \\
\label{locex0}
& \quad \quad \quad \quad \left. \left( 4 (||u_0||_\infty + 1) \left( \frac{||h_\alpha'||_\infty}{2} + \frac{1}{\sqrt{\pi}} \right) \right)^{-2}
\right\} .
\end{align}
In addition, $u$ satisfies $||u||_\infty \leq ||u_0||_\infty + 1$.
\end{prop}

\begin{proof}
Consider the set $\mathcal{S}$ of functions $v : D_{T^*}\to\mathbb{R}$ which satisfy \eqref{IE2} on $D_{T^*}$ and are such that
\begin{equation}
\label{locex1} || v ||_\infty \leq ||u_0||_\infty + 1.
\end{equation}
Next, consider the mapping $M: \mathcal{S} \to \mathbb{R}(D_{T^*})$ given by $M[v]$ for $v\in \mathcal{S}$ where
\begin{align}
\nonumber
M[v](x,t) & = \int_{-\infty}^\infty u_0(s) G(x,t;s,0) ds \\
\label{locex2}
& \quad +  \int_{0}^t \int_{-\infty}^\infty \frac{v^2(s,\tau )}{2} \left( G(x,t;s,\tau )h_\alpha '(s) + G_s(x,t;s,\tau ) h_\alpha (s) \right) ds d\tau
\end{align}
for all $(x,t)\in D_{T^*}$. 
Observe that the first term in the right hand side of \eqref{locex2} is the solution to the heat equation with measurable initial data $u_0 \in L^\infty (\mathbb{R})$, and in particular, is contained in $C^{2,1}(D_{T^*})\cap L^\infty (D_{T^*})$, with bound $||u_0||_\infty$ on $D_{T^*}$.
Also, using \eqref{Gprop1} and \eqref{Gprop2}, it follows that the integrand of the second term on the right hand side of \eqref{locex2} is absolutely integrable, and hence the integral is well-defined, and bounded on $D_{T^*}$, for each $v\in \mathcal{S}$.  
Moreover, via \eqref{Gprop1'}, \eqref{Gprop2'}, \eqref{Gprop1t} and \eqref{Gprop2t}, it follows that the second term in the right hand side of \eqref{locex2} is continuous on $D_{T^*}$ for each $v\in \mathcal{S}$. 
Furthermore, for each $v \in \mathcal{S}$, and all $(x,t) \in D_{T^*}$, observe, via \eqref{Gprop1} and \eqref{Gprop2}, that
\begin{align}
\nonumber 
& \left| \int_{0}^t \int_{-\infty}^\infty \frac{v^2 (s,\tau )}{2} \left( G(x,t;s,\tau )h_\alpha '(s) + G_s(x,t;s,\tau ) h_\alpha (s) \right) ds d\tau \right| \\
\nonumber 
& \quad \leq \frac{||v||_\infty^2}{2} \int_{0}^t \int_{-\infty}^\infty \left( ||h_\alpha'||_\infty G(x,t;s,\tau ) + || h_\alpha ||_\infty |G_s(x,t;s,\tau )|  \right) ds d\tau \\
\nonumber
& \quad \leq ||v||_\infty^2 \left( \frac{||h_\alpha'||_\infty}{2} \sqrt{t} + \frac{1}{\sqrt{\pi}} \right) \sqrt{t} \\ 
\label{locex3}
& \quad \leq 1.
\end{align}
Consequently it follows from \eqref{locex1}-\eqref{locex3} that $M:\mathcal{S}\to\mathcal{S}$.
Now for $v_1,v_2\in\mathcal{S}$ we have 
\begin{align}
\nonumber
& |M[v_1](x,t) - M[v_2](x,t)| \\
\nonumber
& \quad \leq \int_{0}^t \int_{-\infty}^\infty \frac{|v_1^2 (s,\tau ) - v_2^2(s,\tau )|}{2} \left| G(x,t;s,\tau )h_\alpha '(s) + G_s(x,t;s,\tau ) h_\alpha (s) \right| ds d\tau \\
\label{locex4}
& \quad \leq (||v_1||_\infty + ||v_2||_\infty ) \left( \frac{||h_\alpha'||_\infty}{2}\sqrt{t} + \frac{1}{\sqrt{\pi}} \right) \sqrt{t}||v_1 - v_2||_\infty 
\end{align}
for all $(x,t)\in D_{T^*}$, again using \eqref{Gprop1} and \eqref{Gprop2}. 
It follows from \eqref{locex4} and \eqref{locex0} that 
\begin{equation}
\label{locex5}
|| M[v_1] - M[v_2] ||_\infty \leq \frac{1}{2} ||v_1 - v_2||_\infty
\end{equation}
for all $v_1,v_2 \in \mathcal{S}$, and hence, $M$ is a contraction mapping. 
Since the metric space $(\mathcal{S},|| \cdot ||_\infty )$ is complete, it follows that there exists $u^*\in \mathcal{S}$ such that $u^* = M(u^*)$, i.e. $u^*:D_{T^*}\to\mathbb{R}$ is a solution to \eqref{IE1} and \eqref{IE2}, as required.
\end{proof}

\noindent We now establish that the solution $u:D_{T^*}\to\mathbb{R}$ to \eqref{IE1} and \eqref{IE2} given in Proposition \ref{locex} is twice (once) continuously differentiable with repsect to $x$ ($t$), and hence, is a local solution to [IVP]. 
To begin, for a solution $u:D_{T^*}\to\mathbb{R}$ to \eqref{IE1} and \eqref{IE2}, we define the sequence of functions $u_n:D_{T^*}\to\mathbb{R}$ to be

\begin{align}
\nonumber
u_n(x,t) & = \int_{-\infty}^\infty u_0(s) G(x,t;s,0) ds \\ 
\label{locex6}
& \quad  +  \int_{0}^{t_n} \int_{-\infty}^\infty \frac{u^2 (s,\tau )}{2} \left( G(x,t;s,\tau )h_\alpha '(s) + G_s(x,t;s,\tau ) h_\alpha (s) \right) ds d\tau ,
\end{align}
for all $(x,t)\in D_{T^*}$ and $n\in\mathbb{N}$ with $t_n=t-t/{2n}$. 
Observe that for each $n\in\mathbb{N}$, we have $u_n\in C(D_{T^*})\cap L^\infty (D_{T^*})$, via \eqref{Gprop1'}, \eqref{Gprop2'}, \eqref{Gprop1t} and \eqref{Gprop2t}. 
Moreover, as $n\to\infty$, $u_n$ converges to $u$ uniformly on compact subsets of $D_{T^*}$. 
Next we have
\begin{prop}\label{locexHu}
For each $\beta \in (0,1)$, there exists a constant\footnote{Throughout the paper, we denote constants by $c(\cdot , \ldots , \cdot)$, which can change line-by-line, but nonetheless depend only on the quantities listed in brackets.} $c$ such that the solution $u:D_{T^*}\to\mathbb{R}$ to \eqref{IE1} and \eqref{IE2} given in Proposition \eqref{locex} satisfies
\begin{equation} 
\label{locexHu1}
| u(x_1, t) - u(x_2,t) | \leq c(||u_0||_\infty ,\alpha, \beta ) \left( \frac{|x_1-x_2|}{\sqrt{t}} \right)^\beta
\end{equation}
for all $(x_1,t),(x_2,t) \in D_{T^*}$. 
\end{prop}

\begin{proof}
Let $u:D_{T^*}\to\mathbb{R}$ be the solution to \eqref{IE1} and \eqref{IE2} given by Proposition \ref{locex}. 
Then via \eqref{Gprop1'}, for each $\beta \in (0,1)$, it follows that
\begin{equation}
\label{locexHu2} 
\left| \int_{-\infty}^\infty u_0(s) (G(x_1,t;s,0) - G(x_2,t;s,0)) ds \right| \leq ||u_0||_\infty c(\beta) \left( \frac{|x_1-x_2|}{\sqrt{t}} \right)^\beta
\end{equation}
for all $(x_1,t),(x_2,t) \in D_{T^*}$. 
Moreover, it follows from \eqref{Gprop1'}, \eqref{Gprop2'} and \eqref{locex0} that
\begin{align}
\nonumber
& \bigg| \int_0^t \int_{-\infty}^\infty 
\frac{u^2(s,\tau)}{2}[ ( G(x_1,t;s,\tau ) - G(x_2,t;s,\tau))h_{\alpha}'(s) \\
\nonumber 
& \quad \quad \quad \quad  + ( G_s(x_1,t;s,\tau ) - G_s(x_2,t;s,\tau))h_{\alpha}(s)] ds d\tau \bigg| \\ 
\nonumber 
& \quad \leq \frac{(||u_0||_\infty + 1)^2}{2} \int_0^t \left( ||h_\alpha'||_\infty c(\beta)\left( \frac{|x_1-x_2|}{\sqrt{t-\tau}}\right)^\beta + c(\beta)\left( \frac{|x_1-x_2|}{\sqrt{t-\tau}}\right)^\beta \frac{1}{\sqrt{t-\tau}} \right) d\tau \\
\label{locexHu3'}
& \quad \leq  c(||u_0||_\infty , \alpha , \beta ) |x_1-x_2|^\beta {T^*}^{(1-\beta )/2} \\
\label{locexHu3}
& \quad \leq c(||u_0||_\infty, \alpha, \beta)\left(  \frac{|x_1-x_2|}{\sqrt{t}}\right)^\beta
\end{align}
for all $(x_1,t),(x_2,t) \in D_{T^*}$. 
Inequality \eqref{locexHu1} follows from \eqref{locexHu2} and \eqref{locexHu3}, 
as required.

\end{proof}
\noindent Consequently we have

\begin{prop} \label{locexdx}
For the solution $u:D_{T^*}\to\mathbb{R}$ of \eqref{IE1} and \eqref{IE2} given in Proposition \eqref{locex}, $u_x:D_{T^*}\to\mathbb{R}$ exists and $||u_x(\cdot , t)||_\infty$ exists, with $u_x\in C(D_{T^*})$. 
In addition, for each $\beta \in (0,1)$, 
\begin{align} 
\label{locexdx1}
& || u_x(\cdot, t) ||_\infty \leq \frac{c(||u_0||_\infty , \alpha , \beta )}{\sqrt{t}} \quad \forall\ t\in (0,T^*] ; \\
\label{locexdx2}
& | u_x(x_1, t) - u_x(x_2,t) | \leq \frac{c(||u_0||_\infty,\alpha,\beta)}{t^{(1+\beta)/2}} |x_1 - x_2 |^\beta \quad \forall\ t\in (0,T^*].
\end{align} 
\end{prop}

\begin{proof}
For $u_n:D_{T^*}\to\mathbb{R}$ given by \eqref{locex6}, since $u_0$ is measurable and $u_0 \in L^\infty(\mathbb{R})$, $u_{nx}:D_{T^*}\to\mathbb{R}$ exists, is continuous, and is given by 
\begin{align}
\nonumber
u_{nx}(x,t) & = \int_{-\infty}^\infty u_0(s) G_x(x,t;s,0) ds \\ 
\label{locexdx3}
& \quad  +  \int_{0}^{t_n} \int_{-\infty}^\infty \frac{u^2 (s,\tau )}{2} \left( G_x(x,t;s,\tau )h_\alpha '(s) + G_{sx}(x,t;s,\tau ) h_\alpha (s) \right) ds d\tau ,
\end{align}
for all $(x,t)\in D_{T^*}$. 
After a change of variables $s=x+2\sqrt{t-\tau}\lambda$, the second integral in \eqref{locexdx3} can be expressed as
\begin{align}
\nonumber 
& \int_{0}^{t_n} \int_{-\infty}^\infty \frac{u^2 (s,\tau )}{2} G_x(x,t;s,\tau )h_\alpha '(s) ds d\tau \\ 
\label{locexdx4}
& \quad - \int_{0}^{t_n} \int_{-\infty}^\infty  u^2(x+ 2\sqrt{t-\tau}\lambda , \tau ) \frac{\left( \lambda^2 - 1/2\right)}{\sqrt{\pi}(t-\tau )} e^{-\lambda^2}h_\alpha (x+ 2\sqrt{t-\tau}\lambda ) d\lambda d\tau ,
\end{align}
for all $(x,t)\in D_{T^*}$. 
Now, the second integral in \eqref{locexdx4} can be expressed as
\begin{align}
\nonumber
& \int_{0}^{t_n} \int_{-\infty}^\infty  \left( u^2(x+ 2\sqrt{t-\tau}\lambda , \tau ) h_\alpha (x+ 2\sqrt{t-\tau}\lambda ) - u^2(x,\tau ) h_\alpha (x) \right) \\
\label{locexdx5}
& \quad \quad \quad \quad \quad
\times \frac{\left( \lambda^2 - 1/2 \right)}{\sqrt{\pi}(t-\tau )} e^{-\lambda^2} d\lambda d\tau
\end{align}
for all $(x,t)\in D_{T^*}$. 
Using Proposition \ref{locexHu}, $h_\alpha'\in C(\mathbb{R}) \cap L^\infty (\mathbb{R})$, and the mean value theorem, it follows that for each $\beta \in (0,1)$, there exists a constant $c$ such that
\begin{align}
\nonumber
& |u^2(x+ 2\sqrt{t-\tau}\lambda , \tau ) h_\alpha (x+ 2\sqrt{t-\tau}\lambda ) - u^2(x,\tau ) h_\alpha (x) | \\
\nonumber
& \quad \leq 2||u(\cdot , \tau )||_\infty || h_\alpha ||_\infty |u(x+2\sqrt{t-\tau}\lambda , \tau ) - u(x,\tau)| \\
\nonumber
& \quad \quad + ||u(\cdot , \tau )||_\infty^2 |h(x+2\sqrt{t-\tau}\lambda)-h(x)| \\
\label{locexdx4''}
& \quad \leq c(||u_0||_\infty , \alpha, \beta ) \left( \frac{\sqrt{t-\tau}|\lambda|}{\sqrt{\tau}} \right)^\beta 
\end{align}
for all $(x,t;\lambda , \tau) \in \mathcal{X}_{T^*}$.
Therefore, the absolute value of the integral in \eqref{locexdx5} is bounded above by
\begin{align}
\nonumber
& \int_{0}^{t_n} \int_{-\infty}^\infty 
c(||u_0||_\infty , \alpha , \beta) \left( \frac{\sqrt{t-\tau}|\lambda|}{\sqrt{\tau}}\right)^\beta \left| \lambda^2 + 1/2 \right| \frac{1}{(t-\tau )} e^{-\lambda^2} ds d\tau \\
\nonumber
& \quad \leq c(||u_0||_\infty , \alpha , \beta) \int_0^{t_n} \frac{1}{\tau^{\beta / 2} (t-\tau)^{1-\beta / 2}} d\tau \\
\label{locexdx6} 
& \quad \leq c(||u_0||_\infty , \alpha , \beta)
\end{align}
for all $(x,t)\in D_{T^*}$. 
Therefore, via \eqref{locexdx3}, \eqref{locexdx4}, \eqref{locexdx5} and \eqref{locexdx6}, it follows from \eqref{Gprop2} that
\begin{equation}
\label{locexdx7}
|| u_{nx} (\cdot , t) ||_\infty \leq \frac{c(||u_0||_\infty , \alpha , \beta )}{\sqrt{t}}
\end{equation}
for all $t\in (0,T^*]$.
We now demonstrate that $u_{nx}$ converges uniformly on compact subsets of $D_{T^*}$, to a the continuous limit $u_x$, as $n\to\infty$. 
It follows from \eqref{Gprop2} and \eqref{locexdx4''} that
\begin{align}
\nonumber
& \left| \int_{t_n}^{t} \int_{-\infty}^\infty \frac{u^2 (s,\tau )}{2} \left( G_x(x,t;s,\tau )h_\alpha '(s) + G_{sx}(x,t;s,\tau ) h_\alpha (s) \right) ds d\tau \right| \\
\nonumber
& \quad \leq \frac{(||u_0||_\infty + 1)^2}{2} ||h_\alpha'||_\infty \int_{t_n}^t \frac{1}{\sqrt{\pi (t-\tau)}} d\tau \\
\nonumber
& \quad \quad + \bigg| \int_{t_n}^{t} \int_{-\infty}^\infty  \left( u^2(x+ 2\sqrt{t-\tau}\lambda , \tau ) h_\alpha (x+ 2\sqrt{t-\tau}\lambda ) - u^2(x,\tau ) h_\alpha (x) \right) \\
\nonumber
& \quad \quad \quad \quad \quad \quad \quad \times 
 \frac{\left( \lambda^2 - 1/2 \right)}{\sqrt{\pi}(t-\tau )} e^{-\lambda^2} ds d\tau \bigg| \\
\nonumber 
& \quad \leq c(||u_0||_\infty , \alpha) (2n)^{-1/2} \\
\nonumber 
& \quad \quad + \int_{t_n}^t \int_{-\infty}^\infty c(||u_0||_\infty , \alpha , \beta ) \left( \frac{\sqrt{t-\tau}|\lambda|}{\sqrt{\tau}} \right)^\beta |\lambda^2 - 1/2| \frac{1}{(t-\tau )} e^{-\lambda^2} d\lambda d\tau \\
\nonumber
& \quad \leq c(||u_0||_\infty , \alpha , \beta) \left(  (2n)^{-1/2} + \int_{1-1/(2n)}^1 \frac{1}{q^{\beta /2} (1-q)^{1-\beta /2}} dq \right)
\end{align}
for all $(x,t)\in D_{T^*}$.
Therefore, via \eqref{locexdx3}, it follows that $u_{nx}$ is uniformly convergent on (compact subsets of) $D_{T^*}$. 
It thus follows that there exists a continuous limit of $u_{nx}$ on $D_{T^*}$, which coincides with the derivative $u_x$. 
The bound in \eqref{locexdx1} follows immediately from \eqref{locexdx7}.
As a consequence $u_x:D_{T^*}\to\mathbb{R}$ can be represented, alternatively, as 
\begin{align}
\nonumber
u_{x}(x,t) & = \int_{-\infty}^\infty u_0(s) G_x(x,t;s,0) ds \\ 
\nonumber
& \quad  +  \int_{0}^{t} \int_{-\infty}^\infty \frac{u^2 (s,\tau )}{2} \left( G_x(x,t;s,\tau )h_\alpha '(s) + G_{sx}(x,t;s,\tau ) h_\alpha (s) \right) ds d\tau \\ 
\label{locexdx8} 
& = \int_{-\infty}^\infty u_0(s) G_x(x,t;s,0) ds 
- \int_{0}^{t} \int_{-\infty}^\infty (uu_s) (s,\tau ) G_x(x,t;s,\tau )h_\alpha(s) ds d\tau  
\end{align}
for all $(x,t)\in D_{T^*}$.
Finally, from \eqref{locexdx8}, \eqref{locexdx1} and \eqref{Gprop2'} it follows that 
\begin{align}
\nonumber
& |u_x(x_1,t)-u_x(x_2,t)| \\
\nonumber 
& \quad \leq \frac{c(||u_0||_\infty,\beta)}{\sqrt{t}} \left( \frac{|x_1 - x_2 |}{\sqrt{t}}\right)^\beta 
+ \int_{0}^{t} \frac{c(||u_0||_\infty,\alpha , \beta)}{\sqrt{\tau(t-\tau )}} \left( \frac{|x_1-x_2|}{\sqrt{t-\tau}} \right)^\beta d\tau \\
\nonumber
& \quad \leq \frac{c(||u_0||_\infty,\beta)}{\sqrt{t}} \left( \frac{|x_1 - x_2 |}{\sqrt{t}}\right)^\beta 
+ \frac{c(||u_0||_\infty,\alpha , \beta)}{t^{\beta/ 2}} \int_{0}^{1} \frac{1}{\sqrt{q} (1-q )^{(1+\beta)/2}} dq |x_1-x_2|^\beta \\
\label{locexdx9}
& \quad \leq \frac{c(||u_0||_\infty,\alpha,\beta)}{t^{(1+\beta)/2}} |x_1 - x_2 |^\beta 
\end{align}
for all $(x,t)\in D_{T^*}$, from which we arrive at \eqref{locexdx2}, as required. 
\end{proof}

\noindent We can now further extend the regularity in the next result.

\begin{prop} \label{locexdxdxdt}
For the solution $u:D_{T^*}\to\mathbb{R}$ of \eqref{IE1} and \eqref{IE2} given in Proposition \ref{locex}, $u_{xx}:D_{T^*}\to\mathbb{R}$ and $u_t:D_{T^*}\to\mathbb{R}$ both exist and are continuous on $D_{T^*}$.
In addition, for each $\beta \in (0,1)$, 
\begin{equation} 
\label{locexdxdxdt0}
|| u_{xx}(\cdot, t) ||_\infty , || u_{t}(\cdot, t) ||_\infty \leq \frac{c(||u_0||_\infty , \alpha , \beta )}{t} \quad \forall\ t\in (0,T^*] .
\end{equation}
\end{prop}

\begin{proof}
For $u_n:D_{T^*}\to\mathbb{R}$ given by \eqref{locex6}, since $u_0$ is measurable and $u_0 \in L^\infty(\mathbb{R})$, $u_{nxx}:D_{T^*}\to\mathbb{R}$ exists, is continuous, and is given by 
\begin{equation} 
\label{locexdxdx1}
u_{nxx}(x,t) = \int_{-\infty}^\infty u_0(s) G_{xx}(x,t;s,0) ds -  \int_{0}^{t_n} \int_{-\infty}^\infty  (uu_s)(s,\tau) G_{xx}(x,t;s,\tau )h_\alpha (s) ds d\tau ,
\end{equation}
for all $(x,t)\in D_{T^*}$. 
After a change of variables $s=x+2\sqrt{t-\tau}\lambda$, the second integral in \eqref{locexdxdx1} can be expressed as
\begin{equation}
\label{locexdxdx2}
\int_{0}^{t_n} \int_{-\infty}^\infty  (uu_s)(x+ 2\sqrt{t-\tau}\lambda , \tau ) \frac{\left( \lambda^2 - 1/2\right)}{\sqrt{\pi}(t-\tau )} e^{-\lambda^2}h_\alpha (x+ 2\sqrt{t-\tau}\lambda ) d\lambda d\tau ,
\end{equation}
for all $(x,t)\in D_{T^*}$. 
From \eqref{Gprop4} it follows that the second integral in \eqref{locexdx4} can be expressed as
\begin{align}
\nonumber
& \int_{0}^{t_n} \int_{-\infty}^\infty  \left( (uu_s)(x+ 2\sqrt{t-\tau}\lambda , \tau ) h_\alpha (x+ 2\sqrt{t-\tau}\lambda ) - (uu_s)(x,\tau ) h_\alpha (x) \right) \\
\label{locexdxdx3}
& \quad \quad \quad \quad \quad
\times \frac{\left( \lambda^2 - 1/2 \right)}{\sqrt{\pi}(t-\tau )} e^{-\lambda^2} d\lambda d\tau
\end{align}
for all $(x,t)\in D_{T^*}$. 
Using Propositions \ref{locexHu} and \ref{locexdx}, $h_\alpha'\in C(\mathbb{R}) \cap L^\infty (\mathbb{R})$, and the mean value theorem, it follows that for each $\beta \in (0,1)$, there exists a constant $c$ such that
\begin{align}
\nonumber
& |(uu_s)(x+ 2\sqrt{t-\tau}\lambda , \tau ) h_\alpha (x+ 2\sqrt{t-\tau}\lambda ) - (uu_s)(x,\tau ) h_\alpha (x) | \\
\nonumber
& \quad \leq ||u_s(\cdot , \tau )||_\infty || h_\alpha ||_\infty |u(x+2\sqrt{t-\tau}\lambda , \tau , \tau) - u(x,\tau)| \\
\nonumber
& \quad \quad + ||u(\cdot , \tau )||_\infty ||h_\alpha ||_\infty |u_s(x+2\sqrt{t-\tau}\lambda)-u_s(x)| \\
\nonumber
& \quad \quad + ||u(\cdot , \tau )||_\infty ||u_s(\cdot , \tau )||_\infty |h(x+2\sqrt{t-\tau}\lambda)-h(x)| \\
\label{locexdxdx4}
& \quad \leq c(||u_0||_\infty , \alpha, \beta ) \left ( \frac{1}{\sqrt{\tau}}  \left( \frac{\sqrt{t-\tau}|\lambda |}{\sqrt{\tau}} \right)^\beta \right)
\end{align}
for all $(x,t;\lambda , \tau) \in \mathcal{X}_{T^*}$.
Therefore, the absolute value of the integral in \eqref{locexdxdx3} is bounded above by
\begin{align}
\nonumber
& \int_{0}^{t_n} \int_{-\infty}^\infty 
c(||u_0||_\infty , \alpha , \beta) \frac{1}{\sqrt{\tau}} \left( \frac{\sqrt{t-\tau}|\lambda|}{\sqrt{\tau}}\right)^\beta \left| \lambda^2 + 1/2 \right| \frac{1}{(t-\tau )} e^{-\lambda^2} ds d\tau \\
\nonumber
& \quad \leq c(||u_0||_\infty , \alpha , \beta) \int_0^{t_n} \frac{1}{\tau^{(\beta + 1) / 2} (t-\tau)^{1-\beta / 2}} d\tau \\
\label{locexdxdx5} 
& \quad \leq \frac{c(||u_0||_\infty , \alpha , \beta)}{\sqrt{t}}
\end{align}
for all $(x,t)\in D_{T^*}$. 
Therefore, via \eqref{locexdxdx1}, \eqref{locexdxdx2}, \eqref{locexdxdx3} and \eqref{locexdxdx5}, it follows from \eqref{Gprop3} that
\begin{equation}
\label{locexdxdx6}
|| u_{nxx} (\cdot , t) ||_\infty \leq \frac{c(||u_0||_\infty , \alpha , \beta )}{t}
\end{equation}
for all $t\in (0,T^*]$.
It follows, as in the proof of Proposition \ref{locexdx} that $u_{nxx}$ converges uniformly on compact subsets of $D_{T^*}$, to the continuous limit $u_{xx}$, as $n\to\infty$ with the bound on $u_{xx}$ in \eqref{locexdxdxdt0} following immediately from \eqref{locexdxdx6}.
Consequently $u_{xx}:D_{T^*}\to\mathbb{R}$ can be represented, as
\begin{equation}
\label{locexdxdx7} 
u_{xx}(x,t) = \int_{-\infty}^\infty u_0(s) G_{xx}(x,t;s,0) ds -  \int_{0}^{t} \int_{-\infty}^\infty (uu_s)(s,\tau ) G_{xx}(x,t;s,\tau ) h_\alpha (s) ds d\tau  
\end{equation}
for all $(x,t)\in D_{T^*}$. 
Furthermore, $u_n:D_{T^*}\to\mathbb{R}$ given by \eqref{locex6}, since $u_0$ is measurable and $u_0 \in L^\infty(\mathbb{R})$, $u_{nt}:D_{T^*}\to\mathbb{R}$ exists, is continuous, and is given by
\begin{equation} 
\label{locexdt1}
u_{nt}(x,t) = u_{nxx}(x,t) - \int_{-\infty}^\infty  (uu_s)(s,t_n) G(x,t;s,t_n)h_\alpha (s) ds ,
\end{equation}
for all $(x,t)\in D_{T^*}$. 
From \eqref{Gprop1}, it follows that $G(x,t;s,t_n)$ forms a $\delta$-sequence as $n\to\infty$, and since $u$ and $u_x$ are continuous on $D_{T^*}$ it follows that 
\begin{equation} 
\label{locexdt2}
\int_{-\infty}^\infty  (uu_s)(s,t_n) G(x,t;s,t_n)h_\alpha (s) ds \to u(x,t)u_x(x,t)h_\alpha (x)
\end{equation}
for all $(x,t)\in D_{T^*}$. 
Moreover, on any compact subset of $D_{T^*}$ the convergence in \eqref{locexdt2} is uniform. 
Finally, it follows, as in the proof of Proposition \ref{locexdx} that $u_{nt}$ converges uniformly on compact subsets of $D_{T^*}$, to the continuous limit $u_{t}$, as $n\to\infty$.
As a consequence $u_{t}:D_{T^*}\to\mathbb{R}$ is continuous and can be represented, as
\begin{equation}
\label{locexdt3} 
u_{t}(x,t) = u_{xx}(x,t) - u(x,t)u_x(x,t)h_\alpha (x)
\end{equation}
for all $(x,t)\in D_{T^*}$. 
The bound on $u_t$ in \eqref{locexdxdxdt0} now follows from Propositions \ref{locex} and \ref{locexdx}, and \eqref{locexdt3}, as required. 
\end{proof}

\begin{cor} \label{locexIVP}
Let $u:D_{T^*}\to\mathbb{R}$ be the solution to \eqref{IE1} and \eqref{IE2} given in Proposition \eqref{locex}. 
Then $u:D_{T^*}\to\mathbb{R}$ is a solution to [IVP] with $T=T^*$.
\end{cor}

\begin{proof}
From Propositions \ref{locex}, \ref{locexdx} and \ref{locexdxdxdt}, it follows that $u: D_{T^*}\to\mathbb{R}$ satisfies \eqref{IVP1}. 
Moreover, via \eqref{locexdt3} $u:D_{T^*}\to\mathbb{R}$ satisfies \eqref{IVP2}.
Finally, via \eqref{IE1}, $u: D_{T^*}\to\mathbb{R}$ satisfies \eqref{IVP3''}, as required.
\end{proof}

\begin{remk}
\label{remlocex}
Suppose for the solution $u:D_{T^*} \to \mathbb{R}$ to [IVP] constructed in Proposition \ref{locex}, that $u_0\in C^2(\mathbb{R})\cap W^{2,\infty}(\mathbb{R})$. 
Then following the arguments in Propositions \ref{locexHu}, \ref{locexdx} and \ref{locexdxdxdt}, it follows that $u$ can be naturally extended onto $\bar{D}_{T^*}$, with $u(x,0) = u_0(x)$ for all $x \in\mathbb{R}$, and we conclude that $u \in C^{2,1}(\bar{D}_{T^*})$. 
Moreover, $u_x$, $u_{xx}$ and $u_t$ are bounded on $\bar{D}_{T^*}$ by a constant $c(||u_0||_{W^{2,\infty}},\ \alpha )$, which is independent of $t$, recalling \eqref{locex0}. 
\end{remk}

\noindent Before we can establish the existence of global solutions to [IVP] we require \emph{a priori} bounds on solutions to [IVP].

\begin{prop} \label{prop1}
When $u: D_T\to\mathbb{R}$ is a solution to [IVP] then 
\begin{equation*} 
\inf_{x\in\mathbb{R}} u_0 \leq u \leq \sup_{ x\in\mathbb{R}} u_0 \quad \text{ on } D_T. 
\end{equation*}
\end{prop}

\begin{proof}
Let $0 < \epsilon < T$ and $D_{T,\epsilon} = \{ (x,t) \in D_T :\ t\in (\epsilon , T] \}$. 
Via \eqref{IVP2} it follows that $u: \bar{D}_{T,\epsilon}\to\mathbb{R}$ satisfies 
\begin{equation} \label{prop1a}
u_t - u_{xx} + (u h_\alpha )u_x =0 \quad \text{ on } D_{T,\epsilon}.
\end{equation}
Additionally note that there exist positive constants $c(\epsilon ) = O(\epsilon )$ as $\epsilon \to 0^+$, such that 
\begin{equation}
\label{prop1c}
\inf_{x\in\mathbb{R}} u_0 - c(\epsilon) \leq u(x,\epsilon ) \leq \sup_{x\in\mathbb{R}} u_0 + c(\epsilon ) \quad \forall x \in \mathbb{R},
\end{equation} 
via condition \eqref{IVP3''}.
From \eqref{prop1a}, \eqref{IVP1} and \eqref{prop1c}, it follows from the comparison theorem for second order linear parabolic partial differential inequalies (see, for example \cite[Theorem 4.4]{JMDN1}) by considering: 
\begin{align*}
& \overline{u} = \sup_{x\in\mathbb{R}} u_0 + c(\epsilon) \text{ and } \underline{u} = u \text{ on } \bar{D}_{T,\epsilon} ; \\
& \overline{u} = u \text{ and } \underline{u} = \inf_{x\in\mathbb{R}} u_0 - c(\epsilon ) \text{ on } \bar{D}_{T,\epsilon};
\end{align*}
as regular supersolutions and regular subsolutions respectively, that
\begin{equation}
\label{prop1d}
\inf_{x\in\mathbb{R}} u_0 - c(\epsilon) \leq u \leq \sup_{x\in\mathbb{R}} u_0 + c(\epsilon ) \quad \text{ on } \bar{D}_{T,\epsilon}.
\end{equation}
The result follows from \eqref{prop1d} by letting $\epsilon\to 0^+$.
\end{proof}

\noindent We can now establish 

\begin{prop}
\label{globex}
There exists a global solution $u: D_\infty \to\mathbb{R}$ to [IVP].
\end{prop}

\begin{proof}
For any $T>0$, via Proposition \ref{prop1}, any solution to [IVP] is {\emph{a priori}} uniformly bounded on $D_T$. 
Thus, for each $T>0$, it follows from a finite number of applications of Proposition \ref{locex} (with Remark \ref{remlocex}) that there exists a solution to [IVP] on $D_T$, and hence, a global solution to [IVP] exists on $D_\infty$, as required. 
\end{proof}

\noindent We next establish local in time continuous dependence on the initial data, of global solution to [IVP].

\begin{prop}
\label{contdep}
Let $T>0$ and for $i=1,2$, suppose that $u_i : D_T\to\mathbb{R}$ are solutions to [IVP] with constant $\alpha$, and initial data $u_{0i}$, respectively. 
Then,
\begin{equation}
\label{contdep1}
||(u_1-u_2)(\cdot , t) ||_\infty \leq 
||u_{01}-u_{02}||_\infty c(||u_{01}||_\infty , ||u_{02}||_\infty , \alpha , T) \quad \forall t\in (0,T]. 
\end{equation}
\end{prop}

\begin{proof}
Let $0 < \epsilon < T$ and set $v:\bar{D}_{T, \epsilon}\to\mathbb{R}$ to be 
\begin{equation} \label{contdep1'}
v = u_1 - u_2 \quad \text{ on } \bar{D}_{T, \epsilon}.
\end{equation}  
Via \eqref{IE1}-\eqref{IE2}, for given $u_{10}$ and $u_{20}$ there exist constants $c(\epsilon)=  O(\epsilon )$ as $\epsilon \to 0^+$, such that 
\begin{align}
\nonumber 
|v(x,t)| & \leq \int_{-\infty}^\infty |u_{1}(s,\epsilon) - u_{2}(s,\epsilon)| G(x,t;s,0) ds \\ 
\nonumber
& \quad + \frac{1}{2}\int_{0}^t \int_{-\infty}^\infty |u_1^2-u_2^2|(s,\tau + \epsilon ) \\
\nonumber 
& \quad \quad \quad \quad \quad \quad \quad \times \left( G(x,t;s,\tau )|h_\alpha '(s)| + |G_s(x,t;s,\tau )| h_\alpha (s) \right) ds d\tau , \\
\nonumber
& \leq ||u_{01}-u_{02}||_\infty + c(\epsilon ) + \\
\nonumber
& \quad \quad + \frac{1}{2} \int_{0}^t  (||u_1||_\infty + ||u_2||_\infty ) ||v (\cdot ,\tau )||_\infty \left( || h_\alpha'||_\infty + \frac{1}{\sqrt{\pi(t-\tau )}}\right) d\tau , \\
\label{contdep2}
& \leq ||u_{01}-u_{02}||_\infty + c(\epsilon ) + 
\int_{0}^t  \frac{c(||u_1||_\infty, ||u_2||_\infty, \alpha )}{\sqrt{t-\tau}} ||v(\cdot ,\tau )||_\infty d\tau
\end{align}
for all $(x,t)\in D_{T,\epsilon}$. 
We note, via the continuity and bounds on $u_{1t}$ and $u_{2t}$ given in Proposition \ref{locexdxdxdt}, it follows that $|| v(\cdot , t)||_\infty$ is a continuous and bounded function of $t$ on $[\epsilon ,T]$, and hence the integral in \eqref{contdep2} is well-defined. 
It follows immediately that
\begin{equation}
\label{contdep3}
||v(\cdot , t)||_\infty \leq 
||u_{01}-u_{02}||_\infty + c(\epsilon ) 
+ \int_{0}^t  \frac{c(||u_1||_\infty, ||u_2||_\infty, \alpha )}{\sqrt{t-\tau}} ||v(\cdot ,\tau )||_\infty d\tau
\end{equation}
for all $t\in [\epsilon ,T]$. 
Therefore, via a generalisation of Gronwall's inequality (see \cite[Corollary 2]{HYJGYD1}), and the {\emph{a priori}} bounds in Proposition \ref{prop1}, we conclude that
\begin{align}
\nonumber
||v(\cdot , t)||_\infty & \leq (||u_{01}-u_{02}||_\infty + c(\epsilon ))
 \left( c(||u_1||_\infty , ||u_2||_\infty , \alpha) \sum_{n=1}^\infty \frac{t^{n/2}}{\pi^{n/2} n \Gamma (n/2)}\right) \\
\label{contdep3'}
& \leq ( ||u_{01}-u_{02}||_\infty + c(\epsilon )) c(||u_{01}||_\infty , ||u_{02}||_\infty , \alpha , T)
\end{align}
for all $t\in [\epsilon ,T]$. 
On recalling \eqref{contdep1'}, \eqref{contdep1} follows by letting $\epsilon \to 0^+$ in \eqref{contdep3'}, as required.
\end{proof}

\noindent In summary, we have

\begin{thm} \label{GWP}
There exists a unique global solution $u: D_\infty \to \mathbb{R}$ to [IVP]. 
Moreover, for each $T,\epsilon >0$ and Lebesgue measurable $u_{01}\in L^\infty (\mathbb{R})$, there exists $\delta (T, \epsilon , ||u_{01}||_\infty ) >0$ such that for all Lebesgue measurable $u_{02}\in L^\infty (\mathbb{R})$ such that $||u_{01} - u_{02}||_\infty < \delta$ then the corresponding global solutions to [IVP] given by $u_1,u_2: D_T\to\mathbb{R}$ satisfy
\begin{equation*}
||(u_1-u_2)(\cdot , t)||_\infty < \epsilon \quad \forall t\in (0,T].
\end{equation*}
\end{thm}

\begin{proof}
The global existence and uniqueness of solutions to [IVP] follows from Propositions \ref{globex} and \ref{contdep}. 
Local in time continuous dependence also follows from Proposition \ref{contdep}.
\end{proof}

\noindent We conclude this section by establishing some qualitative properties of solutions to [IVP] for initial data of the form \eqref{IVP3}.
First we have

\begin{remk} \label{remk4mono}
Suppose the initial data in Proposition \ref{globex} satisfies $u_0\in C^k(\mathbb{R})\cap W^{k,\infty}(\mathbb{R})$ for some $k\in \mathbb{N}$ with $k\geq 2$. 
Then it follows from Remark \ref{remlocex} that the global solution to [IVP] can be extended continuously onto $\bar{D}_\infty$. 
Moreover, the global solution $u:\bar{D}_\infty\to\mathbb{R}$ has $k$ partial derivatives with respect to $x$ which are continuous on $\bar{D}_\infty$ and bounded on $\bar{D}_T$ for any $T>0$. 
This follows from an induction argument based on the derivative estimates in Propositions \ref{locexdx} and \ref{locexdxdxdt} with the identity 
\begin{equation*}
\frac{\partial^i}{\partial x^i} 
\left( \int_{\mathbb{R}} u_0(s) G(x,t;s,0)ds \right) =
\int_{\mathbb{R}} u_0^{(i)}(s) G(x,t;s,0)ds
\end{equation*} 
for all $(x,t) \in D_\infty$ and $i=1,\ldots , k$, used to bound the first integrals in \eqref{locexdx3} and \eqref{locexdxdx1}. 
As a consequence, it follows that $u_t$ has $k-2$ partial derivatives with respect to $x$ on $\bar{D}_\infty$ which are bounded on $\bar{D}_T$ for any $T>0$.  
\end{remk}

\noindent We now have

\begin{prop}
\label{mono}
Suppose that the initial data for [IVP] is given by \eqref{IVP3} and the corresponding solution is $u: D_\infty\to\mathbb{R}$. 
When $u^-<u^+$ ($u^->u^+$) then $u_x(\cdot , t) > 0$ ($< 0$) for all $t\in (0,\infty )$.
\end{prop}

\begin{proof}
Consider the sequence of functions $u_0^{(n)}:\mathbb{R} \to \mathbb{R}$ for $n\in\mathbb{N}$ such that
\begin{align}
\label{mono1}
& u_0^{(n)}= u_0 \text{ on } \mathbb{R}\setminus [-1/n,1/n] ,\\
\label{mono2}
& u_0^{(n)} \in C^3(\mathbb{R}) \cap L^\infty (\mathbb{R}) ,\\
\label{mono3}
& u_0^{(n)} \text{ are non-decreasing (non-increasing) when } u^-<u^+ (u^->u^+) ,
\end{align} 
and $u_0$ is given by \eqref{IVP3}. 
It follows from Proposition \ref{globex} that there exists a unique solution $u^{(n)}:\bar{D}_\infty \to \mathbb{R}$ to [IVP] with initial data $u_0^{(n)}$. 
Moreover, it follows from \eqref{mono1}-\eqref{mono3}, Proposition \ref{prop1} and Remark \ref{remk4mono} that $w:\bar{D}_\infty\to\mathbb{R}$ given by $w = u_x^{(n)}$ on $\bar{D}_\infty$ satisfies:
\begin{align}
\label{mono4}
& w\in C^{2,1}(\bar{D}_T) \cap L^\infty (\bar{D}_T ) \quad \text{ for each } T>0, \\
\label{mono5}
& w_x,w_t,w_{xx} \in L^\infty (\bar{D}_T) \quad \text{ for each } T>0, \\
\label{mono6}
& w(\cdot , 0) \geq 0\ (\leq 0) \text{ on } \partial D \text{ when } u^- <u^+\ (u^- > u^+), \\
\label{mono7}
& w_t - w_{xx} + u^{(n)}h_\alpha w_x + (h_\alpha w + u^{(n)}h_\alpha' )w = 0 \text{ on } D_\infty .
\end{align}
Properties \eqref{mono4}-\eqref{mono7} ensure that we can apply the minimum (maximum) principle (see \cite[Theorem 3.3]{JMDN1} to $w$ to establish that $w \geq 0\ (\leq 0)$ when $u^-<u^+\ (u^->u^+)$. 

\noindent Recalling \eqref{mono1}-\eqref{mono3}, it follows from Proposition \ref{prop1} that $u^{(n)}$ are uniformly {\emph{a priori}} bounded on $\bar{D}_\infty$ for $n\in\mathbb{N}$. 
Moreover, via Propositions \ref{locexdx} and \ref{locexdxdxdt}, $u_t^{(n)}$ and $u_x^{(n)}$ are bounded on compact subsets of $D_\infty$ uniformly for $n\in\mathbb{N}$. 
Therefore $u^{(n)}$ forms a uniformly bounded equicontinuous sequence of functions on compact subsets of $D_\infty$.
Hence, there exists a subsequence $u^{(n_j)}$ which converges uniformly as $n_j\to \infty$ to a continuous bounded function on each compact subset of $D_\infty$. 
Since the global solution $u:D_\infty \to \mathbb{R}$ to [IVP] with initial data $u_0$ given by \eqref{IVP3} is unique, it follows that on compact subsets of $D_\infty$, $u^{(n_j)}$ converges uniformly to $u$. 
Therefore, $u_x(\cdot, t) \geq 0\ (\leq 0)$ if $u^-<u^+\ (u^->u^+)$.

\noindent Observe from \eqref{IVP3} and \eqref{locexdx8} that $u_x$ is non-constant as $t\to 0^+$. 
Thus from the strong minimum (maximum) principle (see, \cite[Chapter 2]{AF1}) applied to $u_x$ on $[-X,X]\times [T',T]$ with sufficiently small $T'>0$ and arbitratry $X,T>0$, it follows that $u_x>0$ ($<0$) on $D_\infty$, as required. 
\end{proof}

\noindent We next have the far field result

\begin{prop}
\label{xlims}
Suppose the initial data for [IVP] is given by \eqref{IVP3}. 
Then the solution $u: D_\infty\to\mathbb{R}$ satisfies 
\begin{equation*}
u(x,t) \to u^\pm \text{ as } x\to \pm \infty \text{ uniformly for } t\in (0,T].
\end{equation*} 
\end{prop}

\begin{proof}
Let $u$ be the unique global solution to [IVP] with initial condition \eqref{IVP3} and let $\Omega_T := (-\infty , -1] \times [0,T]$. 
Since $u_0$ is continuous on $\mathbb{R}^-$, it follows for that $u$ can be extended onto $\bar{\Omega}_T$ so that $u\in C^{2,1}(\Omega_T )\cap C(\bar{\Omega}_T)$. 
Consider the following pairs of functions $( \underline{u} , \overline{u} )$ with domain $\bar{\Omega}_T$ and co-domain $\mathbb{R}$. 
\begin{itemize}
\item For $u^-<u^+$ and each $(x,t)\in \bar{\Omega}_T$ we define:
\begin{align}
\label{xlims1}
& \underline{u}(x,t) = u^- \text{ and } \overline{u}(x,t) = u(x,t);  \\
\label{xlims2}
&\underline{u}(x,t) = u(x,t) \text{ and } \overline{u}(x,t) = u^- + |u^+-u^-| e^{(||u_0||_\infty + 1)t + x}.
\end{align} 
\item For $u^->u^+$ and each $(x,t)\in \bar{\Omega}_T$ we define: 
\begin{align}
\label{xlims3}
& \underline{u}(x,t) = u^- - |u^+-u^-| e^{(||u_0||_\infty + 1)t + x} \text{ and } \overline{u}(x,t) = u(x,t) ;  \\
\label{xlims4}
&\underline{u}(x,t) = u(x,t) \text{ and } \overline{u}(x,t) = u^-  .
\end{align} 
\end{itemize}
It follows that the pairs in \eqref{xlims1}-\eqref{xlims4} are all regular subsolution and regular supersolutions on $\bar{\Omega}_T$ for the second order linear parabolic partial differential operator $L:C^{2,1}(\Omega_T )\to \mathbb{R}(\Omega_T )$ given by
\begin{equation}
\label{xlims5}
L[w] = w_t - w_{xx} - (uh_\alpha)w_x \text{ on } \Omega_T \quad \forall w\in C^{2,1}(\Omega_T ). 
\end{equation}
It follows from \eqref{xlims1}-\eqref{xlims5}, Proposition \ref{prop1}, and the comparison theorem \cite[Theorem 4.4]{JMDN1} that $u(x,t)\to u^-$ as $x\to - \infty$ uniformly for $t\in (0,T]$ for each $T>0$. 
The corresponding result for the limit as $x\to \infty$ follows from a symmetrical argument.
\end{proof}

\section{Conclusion}
In this note we have established a well-posedness result for [IVP] to complement the large-$t$ asymptotic analysis for solutions to [IVP] contained in \cite{DNJBJMCD1}.
Further work to establish convergence and qualitative properties of the finite difference approximation utilised in \cite{DNJBJMCD1} is of interest to the authors. 
Moreover, the development of methods to rigorously establish more results the theory in \cite{DNJBJMCD1} illustrates, is also of interest to the authors.

\appendix
\section{Properties of $G$} 
\label{Appendix-G}
We note several properties of the fundamental solution to the heat equation on $D_T$, given by \eqref{Gdef}, here:
\begin{align}
\label{Gprop1}
& \int_{-\infty}^ \infty G(x,t;s,\tau) ds = 1 \quad \forall\ (x,t)\in D_T,\ 0\leq \tau <t; \\
\label{Gprop2}
& \int_{-\infty}^ \infty |G_s(x,t;s,\tau)| ds = \frac{1}{\sqrt{\pi(t-\tau )}} \quad \forall\ (x,t)\in D_T,\ 0\leq \tau <t; \\ 
\label{Gprop3}
& \int_{-\infty}^ \infty |G_{ss}(x,t;s,\tau)| ds \leq \frac{c}{(t-\tau )} \quad \forall\ (x,t)\in D_T,\ 0\leq \tau <t; \\ 
\label{Gprop4}
& \int_{-\infty}^ \infty G_{ss}(x,t;s,\tau) ds = 0  \quad \forall\ (x,t)\in D_T,\ 0\leq \tau <t.
\end{align}
Moreover, for any $\beta \in (0,1)$ there exist constants $c(\beta)$ such that:
\begin{align}  
\label{Gprop1'}
& \int_{-\infty}^\infty |G(x_1,t;s,\tau ) - G(x_2,t;s,\tau ) |ds \leq \text{c}(\beta) \left( \frac{|x_1-x_2|}{\sqrt{t-\tau}} \right)^\beta , \\
 \label{Gprop2'}
& \int_{-\infty}^\infty |G_s(x_1,t;s,\tau ) - G_s(x_2,t;s,\tau ) |ds \leq \frac{\text{c}(\beta)}{\sqrt{t-\tau}} \left( \frac{|x_1-x_2|}{\sqrt{t-\tau}} \right)^\beta , \\
 \label{Gprop3'}
& \int_{-\infty}^\infty |G_{ss}(x_1,t;s,\tau ) - G_{ss}(x_2,t;s,\tau ) |ds \leq \frac{\text{c}(\beta)}{(t-\tau)} \left( \frac{|x_1-x_2|}{\sqrt{t-\tau}} \right)^\beta ,
\end{align}
for all $x_1,x_2\in \mathbb{R}$ and $0 \leq \tau < t \leq T$; and
\begin{align}
\label{Gprop1t}
& \int_{-\infty}^\infty |G(x,t_1;s,\tau ) - G(x,t_2;s,\tau ) |ds \leq \text{c}(\beta) \left( \frac{|t_1-t_2|}{t_2 - \tau} \right)^{\beta / 2} ,  \\
\label{Gprop2t}
& \int_{-\infty}^\infty |G_s(x,t_1;s,\tau ) - G_s(x,t_2;s,\tau ) |ds \leq \frac{\text{c}(\beta)}{\sqrt{t_2-\tau}} \left( \frac{|t_1-t_2|}{t_2 - \tau} \right)^{\beta / 2} , \\
\label{Gprop3t}
& \int_{-\infty}^\infty |G_{ss}(x,t_1;s,\tau ) - G_{ss}(x,t_2;s,\tau ) |ds \leq \frac{\text{c}(\beta)}{(t_2-\tau )} \left( \frac{|t_1-t_2|}{t_2 - \tau} \right)^{\beta / 2} ,
\end{align}
for all $x\in \mathbb{R}$ and $0 \leq \tau < t_2 < t_1 \leq T$. 
These properties can be derived via the approach described in \cite[Ch. 1, Lemma 3]{AF1}.

\bibliographystyle{plain} 
\bibliography{Bibliography.bib}

\end{document}